\documentclass[11pt]{amsart}
\font\emailfont=cmtt10

\usepackage{amsmath,amsthm,amsfonts,amscd,flafter,epsf,amssymb,verbatim}
\usepackage{graphics}
\usepackage{epsfig}
\usepackage{psfrag}

\newcommand\commentable[1]{#1}
\newcommand\Rk{\mathrm{rk}}

\newcommand\sd{\mbox{-}}

\newcommand{\lbra}{{\em (}}
\newcommand{\rbra}{{\em )}}

\newtheorem{theorem}{Theorem}[section]

\newtheorem{cor}[theorem]{Corollary}

\newtheorem{lemma}[theorem]{Lemma}

\newtheorem{defn}[theorem]{Definition}

\newtheorem{remark}[theorem]{Remark}
\newtheorem{question}[theorem]{Question}

\def\endproofof{\relax\ifmmode\expandafter\endproofmath\else
  \unskip\nobreak\hfil\penalty50\hskip.75em\hbox{}\nobreak\hfil\bull
  {\parfillskip=0pt \finalhyphendemerits=0 \bigbreak}\fi}
\def\endproofofmath$${\eqno\bull$$\bigbreak}

\def\endproof{\relax\ifmmode\expandafter\endproofmath\else
  \unskip\nobreak\hfil\penalty50\hskip.75em\hbox{}\nobreak\hfil\bull
  {\parfillskip=0pt \finalhyphendemerits=0 \bigbreak}\fi}
\def\endproofmath$${\eqno\bull$$\bigbreak}
\def\bull{\vbox{\hrule\hbox{\vrule\kern3pt\vbox{\kern6pt}\kern3pt\vrule}\hrule}}

\newcommand{\C}{\mathbb{C}}

\newcommand{\Z}{\mathbb{Z}}

\newcommand{\OneHalf}{\frac{1}{2}}

\newcommand{\cm}{\cdot}

\newcommand\SpinC{\mathrm{Spin}^c}

\newcommand\relspinc{\underline{\spinc}}

\newcommand\Filt{\mathcal F}

\newcommand\x{\mathbf x}

\newcommand\z{\mathbf z}

\newcommand\y{\mathbf y}

\newcommand\ModSphere{\ModFlow\left({\mathbb S}\longrightarrow 
\Sym^{g-1}(\Sigma_{1})\times \Sym^2(\Sigma_{2})\right)}
\newcommand\ModSpheres\ModSphere
\newcommand\CF{CF}

\newcommand\CFa{\widehat{CF}}

\newcommand\CFinf{CF^\infty}

\newcommand\HFa{\widehat{HF}}

\newcommand\UnparModSp{\widehat \ModSp}
\newcommand\UnparModFlow\UnparModSp
\newcommand\Mod\ModSp

\newcommand{\spinc}{\mathfrak s}

\newcommand\ModMaps{\mathcal M}
\newcommand\ModSp\ModMaps

\newcommand\spincrel\relspinc

\newcommand\CFK{CFK}
\newcommand\HFK{HFK}

\newcommand\CFKa{\widehat\CFK}

\newcommand\HFKa{\widehat\HFK}

\newcommand\Dual{\mathcal D}
\newcommand\Duality\Dual

\newcommand\ctau[1]{\tau^p(K,{#1})}
\newcommand\ctauprime[1]{\tau^p(K,{#1})}
\newcommand\cs[1]{s^p(K,{#1})}

\newcommand\FiltY{\Filt_\spinc}

\newcommand\Pos{\mathcal{P}}

\newcommand\ons{Ozsv{\'a}th and Szab{\'o}}
\newcommand\os{{Ozsv{\'a}th-Szab{\'o}}}

\commentable{

\title[{On knot Floer homology and cabling II}] 
{On knot Floer homology and cabling II}

}

\author[Matthew Hedden]{Matthew Hedden}
\address{Department of
Mathematics, Massachusetts Institute of Technology, MA \newline
\indent{\emailfont{mhedden@math.mit.edu}}}

\begin{document}

\begin{abstract}
We continue our study of the knot Floer homology invariants of cable knots. For large $|n|$, we prove that many of the filtered subcomplexes in the knot Floer homology filtration associated to the $(p,pn+1)$ cable of a knot,  $K$, are isomorphic to those of $K$.   This result allows us to obtain information about the  behavior of the  \os \ concordance invariant under cabling, which has geometric consequences for the cabling operation.  Applications considered include quasipositivity in the braid group, the knot theory of complex curves, smooth concordance, and lens space (or, more generally, L-space) surgeries.
\end{abstract}

\maketitle
\section{Introduction}

A powerful knot invariant was introduced by \ons \ in \cite{Knots} and independently by Rasmussen in his thesis, \cite{RasThesis}.  The invariant takes the form of the filtered chain homotopy type of a filtered chain complex.  The chain complex is the so-called  \os \ ``infinity" chain complex associated to a $\SpinC$ three-manifold,  $\CFinf(Y,\spinc)$, and the filtration of this chain complex is induced by a knot $K\subset Y$.   Definitions of the chain complex can be found in \cite{HolDisk} and the filtration induced by the knot is defined in \cite{Knots,RasThesis}. Throughout, we will work with null-homologous knots equipped with a fixed Seifert surface, $F$, though more general constructions are possible \cite{RationalSurgeries}. This paper, and its predecessor \cite{Cabling}, study the knot Floer filtration induced by a class of knots called cable knots.

Let $K$ be a knot.  Recall that the $(p,q)$ cable of  $K$, denoted $K_{p,q}$,  is a satellite knot  with pattern the $(p,q)$ torus knot, $T_{p,q}$. More precisely, $K_{p,q}$ is the image of a torus knot living on the boundary of a tubular neighborhood of $K$.\footnote{The depends on an identification of the tubular neighborhood with a solid torus which, in turn, comes from the longitude specified by a Seifert surface.}  Thus $p$ is the number of times $K_{p,q}$ traverses the longitudinal direction of $K$, and $q$ the meridional number.  Throughout, we will assume $p>0$.\footnote{There will be no loss of generality in doing this, since $K_{-p,-q}\simeq -K_{p,q}$ where  $-K_{p,q}$ is $K_{p,q}$ with reversed string orientation. Our invariants are not sensitive to this orientation.}  Our original motivation for studying cable knots and, more generally, satellite knots, lay in the fact that their complements decompose as the union of two three-manifolds joined along a torus, and hence provide a testing ground for the topological quantum field theoretic (TQFT) behavior of the Heegaard Floer invariants in $(2+1)$ dimensions.

The knot Floer filtration is, in fact, a $\Z\oplus\Z$ filtration of $\CFinf(Y,\spinc)$, and it is the filtered chain homotopy type of this $\Z\oplus\Z$ filtration which is the primary knot invariant coming from Heegaard Floer homology.   The existence of two independent $\Z$ filtrations allows one to define many auxilary knot invariants and in this paper we deal with a less robust invariant - the filtered chain homotopy type of the $\Z$-filtration of $\CFa(Y,\spinc)$ obtained by setting one of the $\Z$ filtrations equal $0$.  We denote this filtration by $\FiltY(K)$ so that we have the sequence of inclusions:
$$ 0=\FiltY(K,-i) \subseteq \FiltY(K,-i+1)\subseteq \ldots \subseteq
\FiltY(K,n)=\CFa(Y,\spinc).$$

 \noindent  The associated graded complexes of this filtration, $\frac{\FiltY(K,j)}{\FiltY(K,j-1)}$, will be denoted by $\CFKa_\spinc(Y,K, j)$, and their homology by $\HFKa_\spinc(Y,K, j)$.  The homology groups $\HFKa_\spinc(Y,K,j)$ are commonly referred to as the {\em knot Floer homology groups of $K\subset Y$}.  These groups were studied for the $(p,pn\pm1)$ cables of an arbitrary knot, $K\subset S^3$, in \cite{Cabling,MyThesis}.  In that paper a stabilization theorem was proved which provided a formula for the groups $\HFKa(S^3,K_{p,pn\pm1},i)$ in the case when the parameter $n$ was sufficiently large.  The formula required $H_*(\Filt(K))$ as input, while the output was merely the associated graded object; hence, there was a loss of information.   
Despite this loss of information the formulas and techniques of \cite{Cabling} have proved to be quite useful and, in particular, were implemented by \ons \ \cite{OSThurston} and Ni \cite{NiThurston} in the proof that link Floer homology detects the Thurston norm.

The purpose of this paper is to extend our knowledge of the filtered chain homotopy type of $\Filt(K_{p,pn\pm1})$ beyond the level of its associated graded object.  A primary motivation for this extension comes from the relationship between the filtered chain homotopy type of $\Filt(K)$ and the smooth four-ball genus of $K$, $g_4(K)$.  \ons \ \cite{FourBall} and Rasmussen \cite{RasThesis} define the following numerical invariant of a knot, $K\subset S^3$:
$$\tau(K)=\mathrm{min}\{i\in\Z|H_*(\Filt(K,i))\longrightarrow \HFa(S^3)\ \mathrm{is \ non\-trivial}\}.$$

\noindent It is shown that this invariant provides a lower bound for the four-ball genus,
$$|\tau(K)|\le g_4(K),$$
\noindent and that $\tau$ provides a homomorphism from the the smooth concordance group of knots, $\mathcal{C}$, to $\Z$.   Moreover, the above inequality is sharp for torus knots, providing a new proof of Milnor's famous conjecture \cite{Milnor1968} on the four-genera and unknotting numbers of this family. 

The main result of this paper is the following theorem which, for simplicity, we state for knots in the three-sphere.

\begin{theorem}
\label{thm:filt}
Let $K\subset S^3$ be a knot. Pick any $M\in \Z$. Then there exists a constant
$N>0$ so that $\forall \ n>N$, the following holds for each $j> M$:
$$      H_*(\Filt(K_{p,pn+1},pj+\frac{(pn)(p-1)}{2}-1))\cong H_*(\Filt(K,j-1)). $$
\noindent Furthermore, 
$$\begin{array}{ll}
	H_*(\Filt(K_{p,pn+1},pj+\frac{(pn)(p-1)}{2}-i))\cong &  \\ 
H_*(\Filt(K_{p,pn+1},pj+\frac{(pn)(p-1)}{2}-i-1)) & \forall \ i=2,\ldots,p-1. \\ 
\end{array}$$
In particular, 
$$\tau(K_{p,pn+1})=\left\{\begin{array}{ll}
        p \tau(K)+\frac{(pn)(p-1)}{2}+p-1 & {\text{or}} \\
      p \tau(K)+\frac{(pn)(p-1)}{2}. & \\ 
\end{array}
\right. $$
\end{theorem}

\noindent The theorem has an analogue for $n<0$ (stated in Section \ref{sec:proof}) which we use with the above to prove:

\begin{theorem}
	\label{thm:tau}

Let $K\subset S^3$ be a non-trivial knot, then the following inequality holds for all $n$,
$$  p \tau(K)+ \frac{(pn)(p-1)}{2}\ \  \le\tau(K_{p,pn+1})\le \ \ 
        p\tau(K)+\frac{(pn)(p-1)}{2} +p-1. $$

\noindent In the special case when $K$ satisfies $\tau(K)=g(K)$ we have the equality,
$$  \tau(K_{p,pn+1}) =  p \tau(K)+\frac{(pn)(p-1)}{2},  $$
\noindent whereas when $\tau(K)=-g(K)$ we have
$$  \tau(K_{p,pn+1}) =  p \tau(K)+\frac{(pn)(p-1)}{2} + p-1. $$

\end{theorem}

\begin{remark} Here, and throughout, $g(K)$ denotes the Seifert genus of $K$.  We emphasize that while Theorem \ref{thm:filt} requires the cabling parameter $n$ to be sufficiently large, there is no restriction on $n$ in the statement of Theorem \ref{thm:tau}. 
\end{remark}

Recently, we generalized $\tau(K)$ to a sequence of invariants of a knot $K\subset Y$ in an arbitrary $3$-manifold \cite{tbbounds}.   For each non-vanishing Floer homology class, $\alpha\in \HFa(Y,\spinc)$ we obtain an integer, $\tau_\alpha(Y,K)$.  Theorem \ref{thm:filt} is a special case of Theorem \ref{thm:filtY}, found in the next section.  This latter theorem holds for knots in arbitrary manifolds and we use this more general result to obtain information for each $\tau_\alpha(Y,K)$. 

Of particular interest is the case when $\alpha$ is the \os \ contact invariant, $c(\xi)\in \HFa(-Y)$.  In this case $\tau_\alpha(Y,K)$  provides upper bounds for the classical framing invariants of Legendrian and transverse representatives of $K$ in the contact structure $\xi$.  Indeed, using our present results, Theorem $1.4$ of \cite{tbbounds} was able to provide the first systematic construction of prime knot types in many tight contact structures  whose classical framing invariants are constrained to be arbitrarily negative.

\subsection{Geometric Consequences}
In addition to the theory of Legendrian knots and the connection with the four-ball genus, the Floer invariants of cable knots can be used in several other contexts.  We take some time to discuss these results.  
\subsubsection{Concordance information}
It is straightforward to see that cabling induces a well-defined operation on the smooth concordance group, $\mathcal{C}$. Indeed, if $K$ and $J$ are concordant, their $(p,q)$ cables will be concordant via a concordance which ``follows along" the concordance between $K$ and $J$ (see, for instance, \cite{Kawauchi1980} for more details).   Thus cabling defines a sequence of maps:
$$\phi_{p,q}: \mathcal{C}\rightarrow \mathcal{C},$$
where $\phi_{p,q}([K])=[K_{p,q}]$ (here $[K]$ denotes the smooth concordance class of $K$).   By pre-composing with $\phi_{p,q}$, it follows that any smooth concordance invariant provides a sequence of smooth concordance invariants.  In the present context, we obtain functions $$ \ctau{n}:= \tau \circ \phi_{p,pn+1}([K]) = \tau(K_{p,pn+1}).$$

 Theorem \ref{thm:tau} shows that $\ctau{n}$ is a piecewise linear  function of $n$.  Indeed, the graph of $\ctau{n}$ lies entirely on the lines of slope $\frac{(p-1)(p)}{2}$ whose $y$ intercepts range between $p\tau(K)$ and $p\tau(K)+p-1$.   Moreover, a crossing change inequality for $\tau$ (Equation \eqref{eq:crossing})  shows that $\ctau{n}$ is monotonically decreasing.  This indicates that $\ctau{n}$  has a  finite set of discontinuities: 
$$J_\tau^p(K)= \{ n\in \Z | \ \ctau{n} \ne \ctau{n+1}-\frac{(p-1)(p)}{2} \}.$$
The cardinality of this ``jumping locus"  is at most $p-1$ and may be zero.  For instance, Theorem  \ref{thm:tau} indicates $J_p(K)=\emptyset$ if $\tau(K)=\pm g(K)\ne 0$. On the other hand, for the unknot we have $J_p(\text{unknot})=\{-1\}$.   

One should compare this with recent work of Van Cott \cite{VanCott2008}, which uses formal properties of $\tau$ to reprove the fact that $\ctau{n}$ is bounded between two lines of slope $\frac{(p-1)p}{2}$.  Her results extend to $(p,q)$ cables, but are unable to recover the possible $y$ intercepts of the lines which bound the graph of $\ctau{n}$, showing only that they differ by $p-1$.

In light of this, one might hope that formal properties of $\tau$ could be pushed further to reprove Theorem \ref{thm:tau} without having to understand the Floer chain complexes.   This seems unlikely, due to the fact that the techniques of \cite{VanCott2008} can also be employed in the study of the Rasmussen concordance invariant, $s(K)$.  This latter invariant is defined using Khovanov homology \cite{RasSlice}, and while it shares several important properties of $\tau$ the two invariants are independent \cite{Stau}.   Indeed, we expect the behavior of the corresponding functions $\cs{n}$ to be quite different from that of $\ctau{n}$.  Our intuition comes from the fact that the Alexander polynomial of cable knots is determined by the formula \begin{equation}
\label{eq:satpoly}
\Delta_{K_{p,q}}(t)=\Delta_{T_{p,q}}(t)\cm\Delta_K(t^p).
\end{equation}
On the other hand, there cannot exist a formula which computes the Jones polynomial of cables of $K$ in terms of the Jones polynomial of $K$.  Since knot Floer homology and Khovanov homology categorify the Alexander and Jones polynomial, respectively, we expect any invariant derived from these theories, e.g. $\tau$ and $s$, to have quite different behavior under cabling.  For this reason, we expect invariants derived from $s$ of cables to be very interesting new concordance invariants, and a pursuit of effective means of computation is well-motivated. In fact, one can ask: 

\begin{question} Does  the Rasmussen concordance invariant, applied to all iterated cables of $K$, determine if $K$ is smoothly slice? 
\end{question}
 
This question has many variants obtained by using other satellite operations or asking for more refined concordance information (see, for instance, \cite{Doubling,LN2008}).  One can also ask the question where $s(K)$ is replaced by $\tau(K)$ (or any other smooth concordance invariant).  The results of the present paper, together with further expected TQFT properties of \os \ Floer homology lead us to conjecture that in this case the answer is ``no".

 In a related direction, we point out that the $(p,1)$ cabling operation $\phi_{p,1}$ was studied in detail by Kawauchi in  \cite{Kawauchi1980}. He showed that  $\phi_{p,1}$ is a homomorphism from the algebraic concordance group to itself, and remarked that it does not appear to be so on the level of the (smooth) concordance group in general (see the parenthetical remark at the end of the proof of Proposition $4.1$ of \cite{Kawauchi1980}).  Using Theorem \ref{thm:tau} we are able to prove that this is so.

\begin{cor}\label{cor:concordance}
Let $\mathcal{C}$ denote the smooth concordance group, and let $$\phi_{p,1}: \mathcal{C}\rightarrow \mathcal{C}$$ denote the map  defined by $\phi_{p,1}([K])=[K_{p,1}]$.  Then $\phi_{p,1}$ is not a homomorphism for any $p$
\end{cor}

\bigskip

\subsubsection{Cable knots and complex curves} 
\ Let $V_f$ be a complex curve $$V_f=\{ (z,w)\in \C^2| 0=f(z,w)\in \mathbb{C}[z,w]\},$$ and let $$S^3=\{(z,w)\subset \mathbb{C}^2| |z|^2+|w|^2=1\}$$ be the three-sphere.  Further suppose that 
$$K=V_f \cap S^3$$
\noindent is a transverse intersection.  In this case, $K\subset S^3$ is a knot or link,  and we call knots that arise in this way (transverse) {\em $\C$-knots} (see \cite{Rudolph2005} for a thorough introduction to these knots).  

It is well-known that some iterated cables of the unknot are $\C$-knots.  Indeed, the class of cabled $\C$-knots contains the so-called {\em links of singularities}, which come from complex curves with a single isolated singularity at the origin.  In fact, the links of singularities are precisely the iterated cables of the unknot satisfying a positivity condition, see \cite{EN1985} for a discussion.  A notable feature of the link of a singularity is that its Milnor fiber (a smoothing of the singular complex curve contained in the four-ball \cite{Milnor1968}) can be isotoped into the three-sphere to be a Seifert surface for the knot.  

In light of these classical results,  a natural question to ask is to what extent cabling can be performed in the complex category.  That is, when can a cable knot be a  $\C$-knot? If it is, when is the piece of the complex curve contained in the four-ball  isotopic to a Seifert surface?  To this end, our results provide the following obstructions:

\begin{cor}
	\label{cor:cor1}
Suppose that $K_{p,pn+1}$ is a $\C$-knot. Then $n\ge -2(\frac{\tau(K)}{p-1}+ \frac{1}{p})$. \end{cor}
\begin{cor}\label{cor:cknot}
Suppose  $K_{p,pn+1}$ is a $\C$-knot with defining complex curve, $V_f$.  Further, suppose the genus of the piece of $V_f$ contained in the four-ball is equal to the Seifert genus, $g(K)$.   Then $n\ge 0$ and $\tau(K)=g(K)$.
\end{cor}

Thus, for instance, no negative cable (i.e.  $n<0$) of a knot with $\tau(K)<0$ (e.g. the left-handed trefoil) will ever be a $\C$-knot.  Note, too, that since $\tau(K)\le g_4(K)$, Corollary \ref{cor:cor1} could also be phrased as a (weaker) obstruction which depends solely on the smooth four-genus of $K$. 

Another feature of links of singularities is that they are fibered and, as alluded to above, there is an isotopy taking their fiber surface to their Milnor fiber. Restricting to the category of fibered knots whose fiber surface is isotopic to a piece of a complex curve, we have the following characterization theorem:

\begin{theorem} \lbra Corollary $1.4$ of \cite{ComplexCable}\rbra \label{thm:ComplexCable}
Let  $K$ be a fibered knot. Then $K_{p,q}$ has a Seifert surface which is isotopic to a piece of a complex curve $V_f \cap B^4$ if and only if
\begin{itemize}
\item $K$ has a  Seifert surface which is isotopic to a piece of a complex curve {\em and}
\item $q>0$
\end{itemize}
\noindent In particular, the fiber surface of an iterated cable of the unknot is isotopic to a piece of a complex curve if and only if all the cabling coefficients are positive.
\end{theorem}

Note the lack of restriction on $q$.  We also remark that the ``if" direction of the theorem holds for the  class of knots which bound {\em quasipositive Seifert surfaces}.  Quasipositive Seifert surfaces are those which can be obtained from parallel disks by attaching bands with a positive half twist (see Figure $1$ of \cite{ComplexCable}).   We postpone the proof of Theorem \ref{thm:ComplexCable} until \cite{ComplexCable}.   There, we determines the relationship between the contact structure associated to a fibered knot and those associated to its cables.   Theorem \ref{thm:ComplexCable}  is a corollary of this relationship and of a connection between  contact geometry and the knot theory of complex curves established in \cite{SQPfiber, Rudolph2005}.

\subsubsection{Cable knots and lens space (L-space) surgeries}
One area of low-dimensional topology where the \os \ invariants have had a significant impact is in the study of Dehn surgery \cite{BGH, Ghiggini2007, LensMe, NiFibered, Lens, AbsGrad, Figure8,  RasGoda, Ras2007, Wang2006}.  Many of these results draw on the close relationship between the knot Floer homology invariants of a knot, $K$, and the \os \  invariants of the closed three-manifolds obtained by Dehn surgery on $K$ (see \cite{Knots,IntegerSurgeries, RationalSurgeries}).  In the case that surgery on $K$ yields a manifold with particularly simple Floer homology, this relationship tightly constrains the knot Floer homology invariants.  The knot Floer homology, in turn,  determine various geometric and topological properties of the knot e.g. the genus, fiberedness.    

The three-manifolds with simplest Floer homology are the rational homology spheres, $Y$, for which the rank of the Floer homology is equal to the order of the first (singular) homology, i.e.  rk$\ \HFa(Y) = |H_1(Y;\Z)|$.  These manifolds are called {\em L-spaces}, and the name stems from the fact that lens spaces are L-spaces.    In the case that positive surgery on $K$ yields an L-space, we call $K$ an {\em L-space knot}.   \ons \ show that the knot Floer homology invariants of L-space knots are determined by the Alexander polynomial \cite{Lens}.  In particular, a corollary of their theorem is that the coefficients of the Alexander polynomial of an L-space knot are all $\pm 1$, and that the knot must be fibered \cite{Ghiggini2007,NiFibered} with four-genus equal to the Seifert genus \cite{FourBall}.  Combining Theorem \ref{thm:tau} with their result yields the following obstruction to a cable knot admitting an L-space surgery.

 \begin{cor}
	\label{cor:cor2}
		Suppose that positive surgery on $K_{p,pn+1}$ yields a lens space or, more generally, an L-space.  Then $n\ge 0$ and $\tau(K)=g(K)$.
\end{cor}

  As counterpoint to this obstruction, we also have the following existence theorem:

 \begin{theorem}\label{thm:lspace}
Let $K\subset S^3$ be a non-trivial knot which admits  a positive $L$-space space surgery.  Then $K_{p,q}$ admits positive L-space surgeries whenever  $q\ge p(2g(K)-1)$, 
\end{theorem}

Note  the above result does not require $q$ of the form $q=pn+1$.
Indeed, it is proved somewhat differently from Theorem \ref{thm:filt}, using a standard cut-and-paste topological argument together with known properties of the \os \ invariants.

Theorem \ref{thm:lspace} is interesting in light of a paucity of examples.  The theorem is the first general  construction of  L-space knots outside of the double primitive knots \cite{Berge}.   Indeed, any sufficiently positive iterated cable of a knot which has an actual lens space surgery will itself have L-space surgeries. 

 As noted above, L-space knots have the property that the $\Z\oplus\Z$ filtered chain homotopy type of the knot's filtration of $\CF^\infty(S^3)$ is determined by the Alexander polynomial \cite{Lens}.  Since the Alexander polynomial of a cable knot is determined by Formula \eqref{eq:satpoly}, Theorem \ref{thm:lspace} provides an efficient method for calculating the Floer homology of a large class of cable knots.  For instance, $+5$ surgery on the trefoil is the lens space, $L(5,4)$, and thus the $(p,q)$ cable of the trefoil admits L-space surgeries whenever $q\ge p+1$.  In particular, the Floer homology of the $(2,3)$ cable of the trefoil is determined by its Alexander polynomial which, from Equation \eqref{eq:satpoly} is
$$(t-1+t^{-1})(t^2-1+t^{-2})= t^{3}-t^{2}+1-t^{-2}+t^{-4}.$$
\noindent On the other hand, this is the same Alexander polynomial as that of the $(3,4)$ torus knot.  As this knot also admits lens space surgeries, we find that the two distinct knots have identical Floer invariants.  This example was obtained by a rather lengthy calculation in \cite{MyThesis}.  

We find this example noteworthy as it appears difficult to produce families of L-space knots with the same Floer invariants.  On the other hand, we will show in an upcoming paper that infinite families of knots with identical Floer invariants are rather abundant.   These families, however, do not admit L-space surgeries.  It would be interesting to probe Theorem \ref{thm:lspace} for further examples of L-space knots with identical Floer homology.

\bigskip
\noindent{\bf{Remarks:}}  Versions of \ Theorems \ref{thm:filt} and \ref{thm:tau}  appear in the author's dissertation, \cite{MyThesis}.   

\bigskip
\noindent{\bf{Acknowledgment:}} I wish to thank Peter Ozsv{\'a}th for his advice and encouragement throughout my time as a graduate student, in which the heart of this work was done.  I also thank Chuck Livingston for his interest, and  Cornelia Van Cott for informing me of Kawauchi's results and suggesting Corollary \ref{cor:concordance}.

\section{Proof of Theorems}
\label{sec:proof}
In this section we prove the theorems stated in the introduction.  Several of the proofs rely heavily on the results of \cite{Cabling}. 

\subsection{Notational Background}
Before beginning, we recall a few facts about knot Floer homology.  Our purpose here is to establish notation and is not intended as an introduction to \os \ theory or knot Floer homology.  First, recall that to a knot $K\subset Y$ we can associate a doubly-pointed Heegaard diagram $$(\Sigma,\{\alpha_1,\ldots,\alpha_g\},\{\beta_1,\ldots,\beta_{g} \},w,z).$$
\noindent (see Definition $2.4$ of \cite{Knots}). The \os \ chain complex $\CFa(Y)$ is generated (as a $\Z/2\Z$ vector space)  by $g$-tuples $\x=x_1\times ... \times x_g$ of intersection points, where $x_i\in \alpha_i\cap \beta_{\sigma(i)}$ (here, $\sigma$ is a permutation in the symmetric group on $g$ letters).  The chain complex is equipped with a differential $\partial$ which counts points in moduli spaces of pseudo-holomorphic disks \cite{HolDisk}.  

Let $\mathcal{G}$ be the set of generators.  Using the basepoint, $w$, \ons \ define a map:
$$\spinc_w: \mathcal{G}\rightarrow \SpinC(Y),$$
\noindent where $\SpinC(Y)$ is the set of $\SpinC$ structures on $Y$.  The chain complex splits as direct sum $$\CFa(Y)\cong \underset{\spinc\in \SpinC(Y)}\bigoplus \CFa(Y,\spinc),$$
\noindent with the summand $\CFa(Y,\spinc)$ generated by those $\x\in\mathcal{G}$ with $\spinc_w(\x)=\spinc$.  Using both basepoints, \ons \ define a map:
$$\underline{\spinc}:=\spinc_{w,z}: \mathcal{G}\rightarrow \underline{\SpinC}(Y,K),$$
\noindent where $ \underline{\SpinC}(Y,K)$ is the set of relative $\SpinC$ structures on $Y-K$.  Picking a Seifert surface, $F$, for $K$, we obtain a map:
$$A : \mathcal{G}\rightarrow \Z,$$
\noindent by associating the quantity $A(\x)=\OneHalf\langle c_1(\spincrel(\x)),[F,\partial F]\rangle $ to a generator.  Here $c_1(\spincrel(\x))\in H^2(Y,K;\Z)$ is the relative first Chern class of the relative $\SpinC$ structure.  We will refer to $A$ as the {\em Alexander grading}.  $A$ defines a filtration on $\CFa(Y,\spinc)$, in the sense that $A(\partial (\x))\le A(\x)$ for each $\x\in \mathcal{G}$. Note that $A$ depends on the Seifert surface, but only through its homology class.

Thus, once we pick a homology class of Seifert surface, it makes sense to define $\FiltY(K,i)$ to be the subcomplex of $\CFa(Y,\spinc)$ generated by $\x\in \mathcal{G}$ satisfying $A(\x)\le i$.  We then have the finite length filtration mentioned in the introduction:
$$ 0=\FiltY(K,-i) \subseteq \FiltY(K,-i+1)\subseteq \ldots \subseteq
\FiltY(K,n)=\CFa(Y,\spinc).$$
\noindent We will habitually omit the Seifert surface from the notation and assume throughout that we have chosen a fixed Seifert surface, drawing attention to its role only when there may be some ambiguity.

Finally, we recall the following definition from \cite{tbbounds}.  To state it, let $\alpha\in \HFa(Y,\spinc)$ be a non-vanishing Floer homology class, and let $\iota_m: \FiltY(K,m)\rightarrow \CFa(Y,\spinc)$ be the inclusion map.
\begin{defn} 
	\label{defn:tau_alpha}
$$	\tau_{\alpha}(Y,K)=\mathrm{min}\{m\in\Z|\ \alpha \subset \mathrm{Im} \ (\iota_m)_*\}, $$\end{defn}
\noindent where $(\iota_m)_*$ is the map induced on homology.

\subsection{Proof of Theorem \ref{thm:filt}}
With the above notation, we will prove   the following more general version of  Theorem \ref{thm:filt}:
\begin{theorem}
\label{thm:filtY}
Let $K\subset Y$ be a knot and let  $\spinc\in\SpinC(Y)$.  Pick any $M\in \Z$. Then there exists a constant
$N(M)>0$ so that $\forall \ n>N(M)$, the following holds for each $j> M$:
$$      H_*(\FiltY(K_{p,pn+1},pj+\frac{(pn)(p-1)}{2}-1))\cong H_*(\FiltY(K,j-1)). $$
\noindent Furthermore, 
$$\begin{array}{ll}
	H_*(\FiltY(K_{p,pn+1},pj+\frac{(pn)(p-1)}{2}-i))\cong &  \\ 
H_*(\FiltY(K_{p,pn+1},pj+\frac{(pn)(p-1)}{2}-i-1)) & \forall \ i=2,\ldots,p-1. \\ 
\end{array}$$
In particular, 
$$\tau_\alpha(Y,K_{p,pn+1})=\left\{\begin{array}{ll}
        p \tau_\alpha(Y,K)+\frac{(pn)(p-1)}{2}+p-1 & {\text{or}} \\
      p \tau_\alpha(Y,K)+\frac{(pn)(p-1)}{2}. & \\ 
\end{array}
\right. $$
\end{theorem}

\begin{remark}  In the above, the filtration for $K_{p,pn+1}$ is defined using a Seifert surface $F_{p,pn+1}$ satisfying $[F_{p,pn+1}]=p[F]$, where $F$ is the Seifert surface used to define the filtration for $K$.
\end{remark}
\begin{proof} The key tool in proving the theorem is

\begin{lemma}
\label{lemma:hd}\lbra Lemma 2.2 of \cite{Cabling}\rbra
Let
$$H=(\Sigma,\{\alpha_1,\ldots,\alpha_g\},\{\beta_1,\ldots,\beta_{g-1},\mu\},z,w),$$
be a Heegaard diagram for a knot K, where $\mu$ is a meridian for $K$.   Then 
$$H(p,n)= (\Sigma,\{\alpha_1,\ldots,\alpha_g\},\{\beta_1,\ldots,\beta_{g-1},\tilde{\beta}\},z',w), $$

\noindent is a Heegaard diagram for
$K_{p,pn+1}$.  $H$ and $H(p,n)$ differ only in the final curve  $\tilde{\beta}$.  Here, $\tilde{\beta}$ is obtained by winding
$\mu$  around an $n$-framed longitude for the knot $(p-1)$ times. $w$
is to remain fixed under this operation. $z$ is replaced by a
basepoint $z'$ so that the arc
connecting $z'$ and $w$ has intersection number $p$
with $\tilde{\beta}$.  \lbra See
Figure ~\ref{fig:wind}.\rbra
\end{lemma}

\begin{figure}
\psfrag{p}{$p$}
\psfrag{x0}{{\small $x_0$}}
\psfrag{x1}{{\small$x_1$}}
\psfrag{x2}{{\small$x_2$}}
\psfrag{x3}{{\small$x_3$}}
\psfrag{x4}{{\small$x_4$}}
\psfrag{x5}{{\small$x_5$}}
\psfrag{x6}{{\small$x_6$}}
\psfrag{x7}{{\small$x_7$}}
\psfrag{x8}{{\small$x_8$}}
\psfrag{z}{$z$}
\psfrag{w}{$w$}
\psfrag{z'}{$z'$}
\psfrag{ai}{$\alpha_i$}
\psfrag{aj}{$\alpha_j$}
\psfrag{ag}{$\alpha_g$}
\psfrag{H(p,n)}{$H(p,n)$}
\psfrag{H}{$H$}
\psfrag{mu}{$\mu$}
\psfrag{lambda}{$\lambda$}
\psfrag{betatilde}{$\tilde{\beta}$}

\begin{center}
\includegraphics{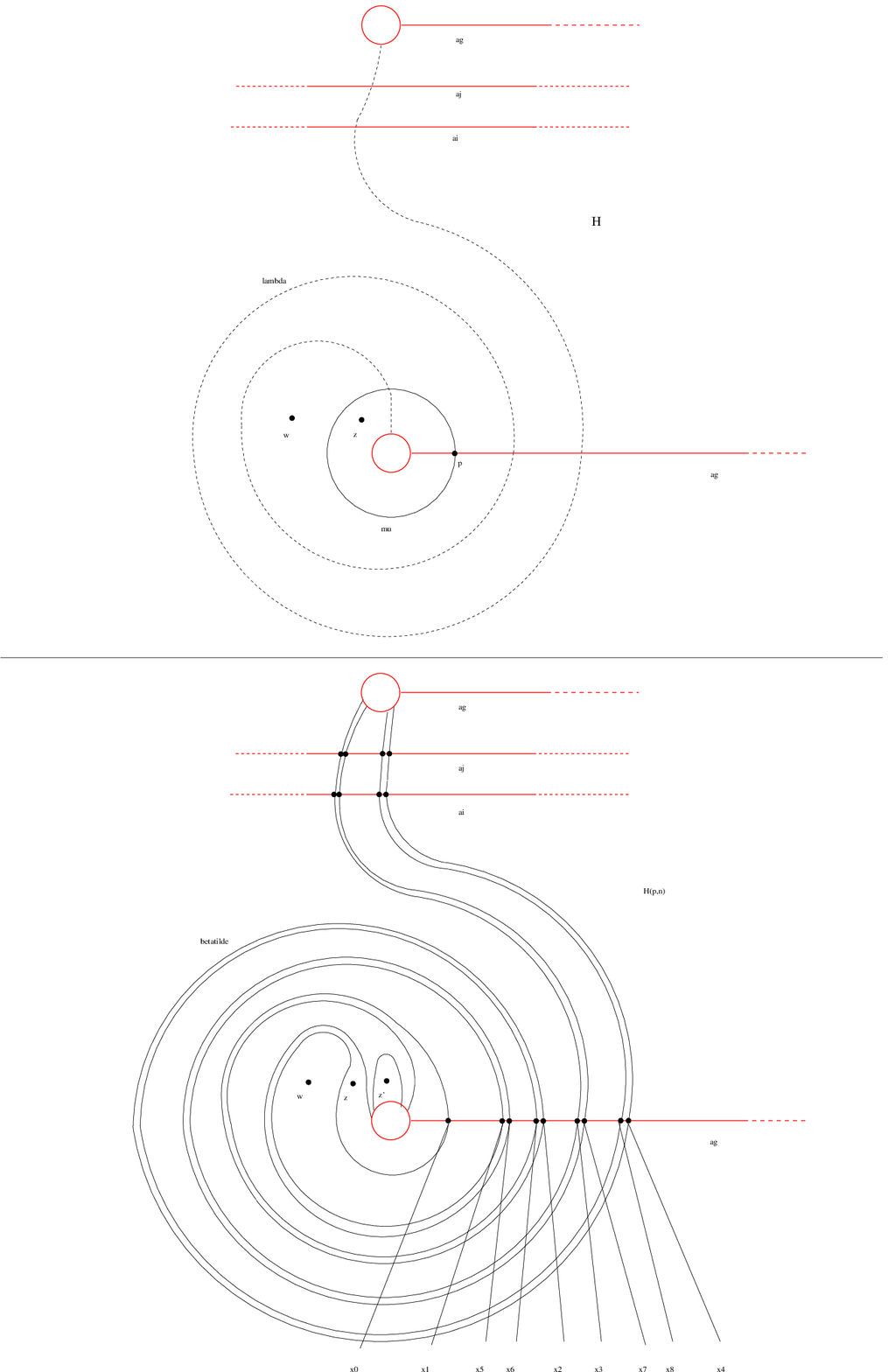}
\caption{\label{fig:wind}   Shown are parts of the diagrams $H$ and $H(p,n)$  (here, $p=3$ and $n=2$). On $H(p,n)$, we have slid the meridian, $\mu$, $2$ times along a $2$-framed longitude $\lambda$. This figure establishes the ordering on exterior intersection points in terms of the ordering on $x_i\in \alpha_g\cap \tilde{\beta}$.   }
\end{center}
\end{figure}

  Lemma 3.1 of \cite{Cabling} shows that with the addition of a third basepoint (which, by an abuse of notation, we also denote by $z$) $H(p,n)$ specifies $K$ by using the pair $(w,z)$ while it specifies $K_{p,pn+1}$ using the pair $(w,z^\prime)$. See Figure \ref{fig:wind}.  Increasing the parameter $n$ for the cable has the effect of adding many generators to the knot Floer chain complexes derived from $H(p,n)$.  These generators, $\x'$, were called {\em exterior intersection points} in the predecessor, and are characterized by the property of having one component of $\x'$ of the form $x_i\in \alpha_g\cap \tilde{\beta}$, with $x_i$ lying in a small neighborhood of the meridian, $\mu$.  The exterior points are in $(2n(p-1)+1)$-to-one correspondence with the generators of the  chain complex for $K$ which came from $H$, i.e. there is $(2n(p-1)+1)$-to-one map:
$$\pi:\{\mathrm{exterior \ intersection \ points \ of \ H(p,n)} \}\longrightarrow \{\mathrm{intersection \ points \ of\ H}\}.$$

More precisely, any generator for the Heegaard diagram $H$ is of the from $\{p,\y\}$, where $p\in \alpha_g\cap \mu$ and $\y$ is a $(g-1)$-tuple of intersection points.  The fibers of $\pi$ are given by 
$$\pi^{-1}(\{p,\y\}) = \{ \ \{x_i,\y\} \ |  \ x_i\in \alpha_g\cap \tilde{\beta}, \ i=0,..., 2n(p-1)\}.$$
Figure \ref{fig:wind} establishes an ordering for the fiber $\pi^{-1}(\{p,\y\})$, in terms of an ordering on the $x_i$.  In terms of this ordering, we call intersection points of the form $\{x_0,\y\}$ {\em outermost intersection points}.  Let us establish some notation:
 $$ C(i):= \{ (g-1)\sd \text{tuples},\ \y \ | \ \{p,\y\} \text{\ has\ Alexander\ grading\ } i \text{\ for\  }H  \} $$
 $$ A:= \text{\ Alexander\ grading\ for\ generators\ from\   } H(p,n),\text{\ with\ respect\ to\ } (w,z)$$
$$ A':= \text{\ Alexander\ grading\ for\ generators\ from\        } H(p,n),\text{\ with\ respect\ to\  } (w,z')$$
 
We will determine the Alexander gradings, $A$ and $A'$,  of exterior intersection through a sequence of lemmas.  The first is an adaptation of Lemma $3.6$ of \cite{OSThurston} to the present notation.   It determines $A$ and $A'$ for the outermost intersection points.  

\begin{lemma}\label{lemma:pindown}
Let $\x=\{x_0,\y\}$ be an outermost intersection point with $\y\in C(i)$.  Then
\begin{enumerate}
\item The $\SpinC$ structure $\spinc\in \SpinC(Y)$ associated to $\x$ is independent of  $z$ or $z'$, and agrees with the $\SpinC$ structure associated to the corresponding generator for $H$, $\{p,\y\}$.
\item $A(\{x_0,\y\})=i$.  That is, the Alexander grading of $\{x_0,\y\}$  with respect to $(w,z)$ agrees with that of $\{p,\y\}$.
\item  $A'(\{x_0,\y\})= pA(\{x_0,\y\})+\frac{(pn)(p-1)}{2} $
\end{enumerate}
\end{lemma} 
\begin{proof}  The first two parts follow from the fact that $(1)$ the map  $\spinc_w$ does not depend on $z$ or $z'$ and $(2)$  there is an isotopy taking $\tilde{\beta}$ to $\mu$ in the complement of $z$, under which the generator 
$\{x_0,\y\}$  becomes identified with $\{p,\y\}$.  The third is the content of Lemma $3.6$ of \cite{OSThurston} which establishes the result in the context of multi-pointed Heegaard diagrams for links.  One can pass from their result to the present case by setting $l=1$,  $p_1=p$, and $q_1=pn+1$ and then considering Chern classes of relative  $\SpinC$ structures to determine the Alexander grading.
\end{proof}

Next, we have Lemmas $3.3$ and $3.4$ of \cite{Cabling}
\begin{lemma}
Let $H(p,n)$ be as above.  Then for odd integers $i < 2n$, we have 
\begin{gather*}
A(x_{i-1} ,\y)-A(x_{i},\y)=A'(x_{i-1},\y)-A'(x_{i},\y)=1\\
A(x_i,\y)-A(x_{i+1},\y)=0\\
A'(x_i,\y)-A'(x_{i+1},\y)=p-1.
\end{gather*}
\end{lemma}

\begin{lemma}
\label{lemma:varytuple}
Suppose $\y \in C(j), \z \in C(k)$.  Then,
\begin{gather*}A(x_i,\y)-A(x_i,\z)=j-k\\
A'(x_i,\y)-A'(x_i,\z)=p(j-k).\end{gather*}
\end{lemma}

  Lemma 3.5 of \cite{Cabling}  shows that, by making the cabling parameter $n$  sufficiently large, we can ensure that  exterior intersection points generate the highest $A$ and $A'$ gradings.  More precisely, we have the following restatement of Lemma 3.5 of \cite{Cabling}:
\begin{lemma}
\label{lemma:isolate}
Pick $l\in \Z$, and let 
$$g=\mathrm{max}\{ A(\x) \ | \ \x \text{\ is\ any\ intersection\ point,\ exterior\ or\ not} \}.$$
  Then there exists a constant  $N > 0$ such that for all $n$ with $n> N$,  the only intersection points with $A(\x)\ge g-l$ are exterior. 
\end{lemma}

\noindent The lemmas are summarized by the table and caption in Figure \ref{fig:chaincomplex}.

\begin{figure}

\psfrag{1}{\small $ C(g)$}
\psfrag{a}{\small $x_0$ }
\psfrag{2}{\small $ C(g-1)$}
\psfrag{b}{\small $x_1$}
\psfrag{3}{\small $C(g-2)$}
\psfrag{c}{\small $x_2$}
\psfrag{4}{$\ldots$}
\psfrag{5}{\small $C(-g)$}
\psfrag{d}{\small $x_3$}
\psfrag{e}{\small $x_4$}
\psfrag{f}{\small $\vdots$}
\psfrag{g}{\small $x_5$}
\psfrag{a1}{\tiny $(g,g')$}
\psfrag{a2}{\tiny $(g-1,g'-p)$}
\psfrag{a3}{\tiny $(g-2,g'-2p)$}
\psfrag{a5}{\tiny $(g-2d,g'-2pd)$}
\psfrag{b1}{\tiny $(g-1, g'-1)$}
\psfrag{b2}{\tiny $(g-2,g'-p-1)$}
\psfrag{b3}{\tiny $(g-3,g'-2p-1)$}
\psfrag{b5}{\tiny $(g-2d-1,g'-2pd-1)$}
\psfrag{c1}{\tiny $(g-1,g'-p)$}
\psfrag{c2}{\tiny $(g-2,g'-2p)$}
\psfrag{c3}{\tiny $(g-3,g'-3p)$}
\psfrag{c5}{\tiny $(g-2d-1,g'-2pd-p)$}
\psfrag{d1}{\tiny $(g-2,g'-p-1)$}
\psfrag{d2}{\tiny $(g-3,g'-2p-1)$}
\psfrag{d3}{\tiny $(g-4,g'-3p-1)$}
\psfrag{d5}{\scriptsize $(g-2d-2,g'-2pd-p-1)$}
\psfrag{e1}{\tiny $(g-2,g'-2p)$}
\psfrag{e2}{\tiny $(g-3,g'-3p)$}
\psfrag{e3}{\tiny $(g-4,g'-4p)$}
\psfrag{e5}{\tiny $(g-2d-2,g'-2pd-2p)$}
\psfrag{g1}{\tiny $(g-3,g'-2p-1)$}
\psfrag{g2}{\tiny $(g-4,g'-3p-1)$}
\psfrag{g3}{\tiny $(g-5,g'-5p-1)$}
\psfrag{g5}{\tiny $(g-2d-n,g'-2pd-np)$}

\includegraphics{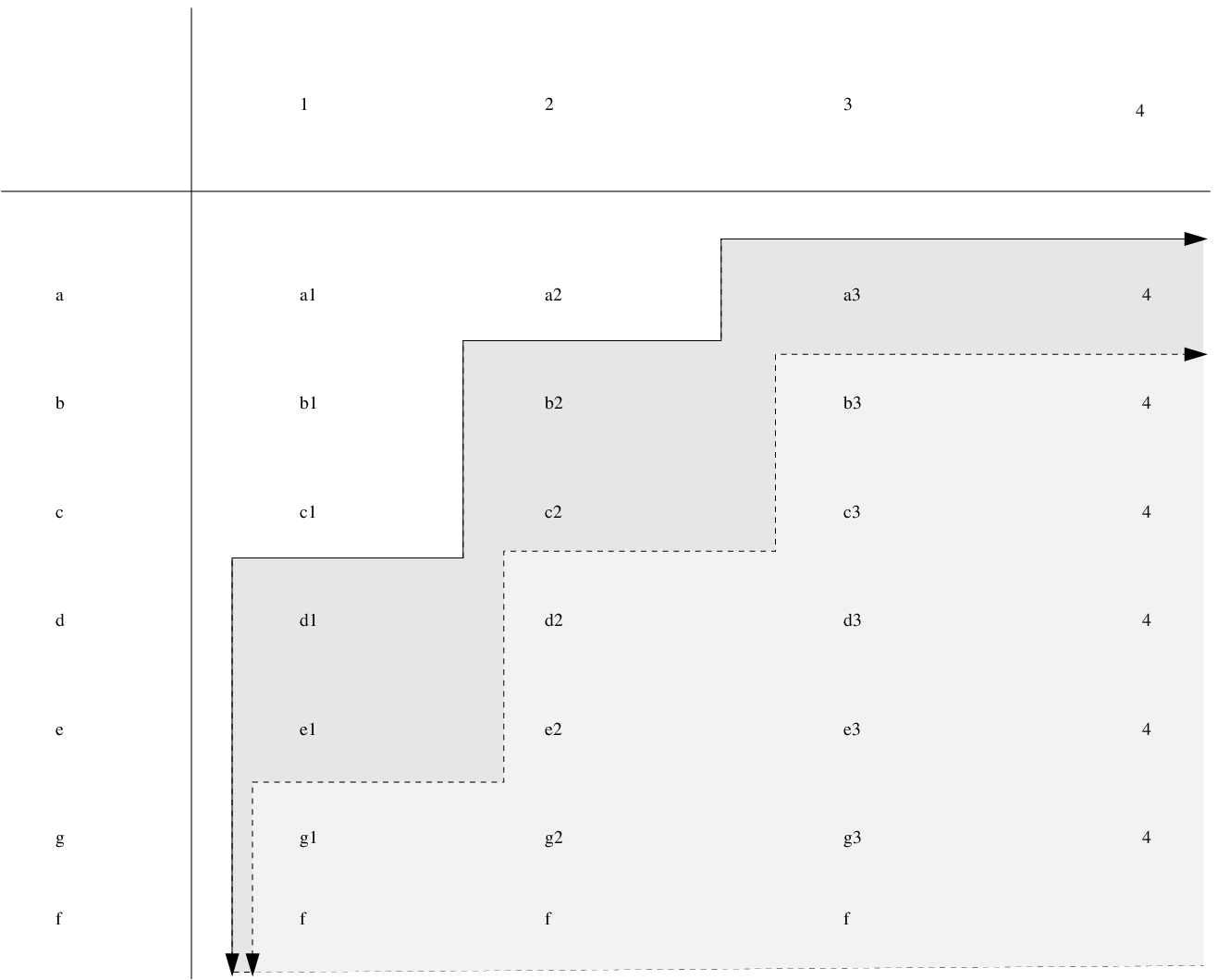}

\caption{\label{fig:chaincomplex}
Table of $A$ and $A'$ gradings of exterior points for the Heegaard diagram, $H(p,n)$.   For each ordered pair, the number on the left is the $A$ grading (i.e. the companion knot's grading). The number on the right is $A'$ (the cabled
  knot's grading).  According to Lemma \ref{lemma:pindown}, we have $g'=pg+\frac{(pn)(p-1)}{2}$. The area below the top (solid) line represents the chain complex $\FiltY(K_{p,pn+1},g'-p-1)=\FiltY(K,g-2)$.  The area below the bottom (dashed) line represents the chain complex $\FiltY(K_{p,pn+1},g'-2p-1)=\FiltY(K,g-3)$. }

\end{figure}

Theorem \ref{thm:filt} follows quickly from the table.  The key point is  that $H(p,n)$  is a diagram for $Y$, regardless of whether we use the basepoint $z$ or $z^\prime$. The only restriction on the differential for the filtered subcomplexes is that $n_w(\phi)=0$, which is independent of $z$ and $z^\prime$.  In light of these remarks and the table of filtrations, we have the following isomorphisms of  {\em chain complexes}, for each $\spinc\in \SpinC(Y)$ and $k<l$
$$ \begin{array}{ll} \FiltY(K_{p,pn+1},pg+\frac{(pn)(p-1)}{2}-p(k-1)-1))& = \FiltY(K,g-k)   \\
	\FiltY(K_{p,pn+1},pg+\frac{(pn)(p-1)}{2}-p(k-1)-i)& = \FiltY(K_{p,pn+1},pg+\frac{(pn)(p-1)}{2}-p(k-1)-i-1)) \\
&	\forall \ i=2,\ldots,p-1  \\ \end{array}$$

Note the appearance of the $\SpinC$ structure, $\spinc$.  Up to this point we had not distinguished between intersection points corresponding to different $\SpinC$ structures; indeed, the table represents all intersection points.  However, it is straightforward to see that the table splits as a direct sum of complexes according to $\SpinC$ structures on $Y$, yielding the above.  This follows from the fact that  part $(1)$ of Lemma \ref{lemma:pindown} actually applies to any exterior point in the table which, in turn, follows from the fact that $\{x_j,\y\}$ is connected to $\{x_0,\y\}$ by a Whitney disk (c.f the proof of Lemmas $3.3$ and $3.4$ of \cite{Cabling}).

Taking $l>g-M$ in Lemma \ref{lemma:isolate} and changing variables $j-1=g-k$ yields the first part of the theorem.  For the second part we let $M<-g(K)$, where $g(K)$ is the genus of $K$.  It follows from the adjunction inequality for knot Floer homology that 

$$\begin{array}{ll} H_*(\FiltY(K,j))=0  & \forall j<-g(K). \end{array}$$
	
\noindent The second part of the theorem now follows from the definition of $\tau_\alpha$. \ \ \ \ \ \ \ \ \ \ \ \ \ \ \ \ \ \ \ \ \ \ \ \ \ \ \ $\square$

Examining the Heegaard diagram $H(p,n)$ when $n<0$, we are  lead to:

\newpage
\begin{theorem}
\label{thm:filtYneg}
Let $K\subset Y$ be a knot and let  $\spinc\in\SpinC(Y)$.  Pick any $M\in \Z$. Then there exists a constant
$N(M)>0$ so that $\forall \ n>N$, the following holds for each $j< M$:
$$      H_*(\FiltY(K_{p,-pn+1},pj-\frac{(pn)(p-1)}{2}-1))\cong H_*(\FiltY(K,j-1)). $$
\noindent Furthermore, 
$$\begin{array}{ll}
	H_*(\FiltY(K_{p,-pn+1},pj-\frac{(pn)(p-1)}{2}+i))\cong &  \\ 
H_*(\FiltY(K_{p,-pn+1},pj-\frac{(pn)(p-1)}{2}+i+1)) & \forall \ i=1,\ldots,p-2. \\ 
\end{array}$$
In particular, 
$$\tau_\alpha(Y,K_{p,-pn+1})=\left\{\begin{array}{ll}
        p \tau_\alpha(Y,K)-\frac{(pn)(p-1)}{2}+p-1 & {\text{or}} \\
      p \tau_\alpha(Y,K)-\frac{(pn)(p-1)}{2}. & \\ 
\end{array}
\right. $$
\end{theorem}

\subsection{Proof of Theorem \ref{thm:tau}}

We now turn to the proof of Theorem \ref{thm:tau}. We focus first on proving the inequality:
\begin{equation} \label{eq:tauinequality} p \tau(K)+ \frac{(pn)(p-1)}{2}\ \  \le\tau(K_{p,pn+1})\le \ \ 
     p \tau(K)+\frac{(pn)(p-1)}{2} +p-1. 
        \end{equation}
        
  This will follow rather quickly from   Theorems \ref{thm:filtY} and \ref{thm:filtYneg}, together with a crossing change inequality satisfied by $\tau$ to interpolate between the cases $n>0$ and $n<0$.   
Recall that two knots $K_+, K_- \subset S^3$ are said to differ by a crossing change if there exists an embedded three-ball $B^3\subset S^3$, outside of which the knots agree,$$(S^3,K_+)\backslash B^3 \cong (S^3,K_-)\backslash B^3,$$
\noindent and such that the local pictures of $(B^3, B^3\cap K_+)$ (resp. $(B^3, B^3\cap K_-)$) are given by Figure \ref{fig:crossingchange}. 
\begin{figure}
\psfrag{K+}{$K_+$}
\psfrag{K-}{$K_-$}

\begin{center}
\includegraphics{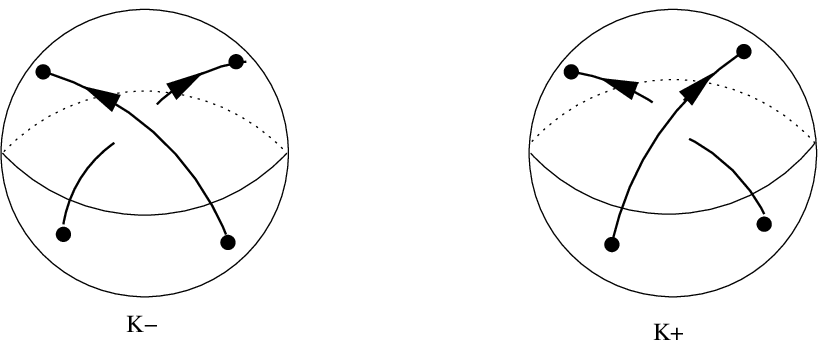}
\caption{\label{fig:crossingchange} A crossing change. $K_+,K_-\subset S^3$ agree, except in the three-ball pictured here.}
\end{center}
\end{figure}

Corollary $1.5$ of \cite{FourBall} states that if $K_+$ and $K_-$ differ by a crossing change, then we have the following inequality:
\begin{equation}
\label{eq:crossing}
\tau(K_+)-1\le \tau(K_-) \le \tau(K_+).
\end{equation}
Now it is straightforward to see that $K_{p,pl+1}$ can be changed into $K_{p,p(l-1)+1}$ by a sequence of $\frac{p(p-1)}{2}$ crossing changes, each of which changes a positive crossing to a negative (to see this induct on $p$ and, for the induction step, change the first $p-1$ crossings of a full twist on $p$ strands).   Thus
\begin{equation} \label{eq:ninequality} \tau(K_{p,pl+1}) -\frac{p(p-1)}{2} \le \tau(K_{p,p(l-1)+1})\le \tau(K_{p,pl+1}).
\end{equation}
Theorems \ref{thm:filtY} and \ref{thm:filtYneg} tell us that Inequality \eqref{eq:tauinequality} is satisfied for $K_{p,pn+1}$ provided $|n|>N$.  Combining this with Inequality \eqref{eq:ninequality} yields \eqref{eq:tauinequality} for all $n$.

\begin{remark} The proof of Inequality \eqref{eq:tauinequality}  will easily extend to each of the invariants $\tau_\alpha(Y,K)$, once a generalization of the crossing change inequality is established for knots in $Y$.  We defer the proof of this latter inequality to \cite{Slice}, where it will follow from the fact that $|\tau_\alpha(Y,K)|$ bounds the  genus of any smoothly embedded surface  $$i: (F,\partial F) \hookrightarrow Y\times [0,1],$$ with $i|_{\partial F} = K\subset Y\times \{1\}$.
\end{remark}

We turn now to the case when $\tau(K)=\pm g(K)$.  Consider the function defined in the introduction:
$$\ctauprime{n}:=\tau(K_{p,pn+1})$$ 

Inequality \eqref{eq:tauinequality} says that the graph of $\ctauprime{n}$ is bounded between the parallel lines \newline $y^{\pm}(n)=\frac{(pn)(p-1)}{2} + c^{\pm},$ where $$\begin{array}{ll} c^+= & p \tau(K) + p-1 \\  c^-= & p \tau(K) \end{array}$$

On the other hand, Theorem \ref{thm:filtY} tells us that for $n>N$, $\ctauprime{n}$ is either $y^+(n)$ or $y^-(n)$.  Similarly, Theorem \ref{thm:filtYneg} says that $\ctauprime{n}$ agrees with either $y^+(n)$ or $y^-(n)$, provided $n<-N$.

Next, we have a lemma:
 \begin{lemma} Suppose $\ctauprime{n}=y^+(n)$ for all $n>N$.  Then $\ctauprime{n}=y^+(n)$ for all $n$.  
 
\noindent  Likewise, suppose $\ctauprime{n}=y^-(n)$ for all $n<-N$.  Then $\ctauprime{n}=y^-(n)$ for all $n$.
 \end{lemma}
 \begin{proof} The lemma follows easily from the fact $\ctauprime{n}$ is bounded between $y^\pm(n)$ for all $n$, together with Inequality \eqref{eq:ninequality}.  More precisely, Inequality \eqref{eq:ninequality} says that $$\ctauprime{n}-\ctauprime{n-1}\le \frac{p(p-1)}{2},$$ \noindent for all $n$. However, if $\ctauprime{n}=y^+(n)$ for all $n>N$, then the only way for $\tau^p$ to be bounded by $y^+$ (for all $n$) is if $\tau^p$ decreases by exactly $\frac{p(p-1)}{2}$ each time we decrease $n$ by one. Similar considerations hold if $\ctauprime{n}=y^-(n)$ for all $n<-N$.
 \end{proof}
 
We will show that if $\tau(K)=-g(K)$, then  $\ctauprime{n}=y^+(n)$ for all $n>N$.  Similarly, if $\tau(K)=g(K)$ then  $\ctauprime{n}=y^-(n)$ for all $n<-N$.  To this end, we have:

\begin{theorem}
\label{thm:largen} \lbra Theorem $1.2$ of \cite{Cabling}\rbra \  Let $K\subset S^3$ be a knot of genus $g$. Pick any $M\in \Z$. Then there exists a constant
$N(M)>0$ so that $\forall \ n>N$, the following holds for each $j< M$: 
$$ \HFKa_*(K_{p,pn+1},i)\cong 
\left\{\begin{array}{ll}

 H_{*+2(j-g)}(\Filt(K,j-g)) & {\text{for
 $i=pg+\frac{(p-1)(pn)}{2}-pj$}} \\

 H_{*+2(j-g)+1}(\Filt(K,j-g)) & {\text{for $i=pg+\frac{(p-1)(pn)}{2}-pj-1$}} \\

0 & {\text{otherwise.}}\\
\end{array}
\right. $$
\end{theorem}

Considering $j=2g$ in the above theorem shows that 
$$\HFKa(K_{p,pn+1},-pg+\frac{(p-1)(pn)}{2})\cong \Z_{(-2g)}$$
\noindent This shows that $\tau(K_{pn+1})\ne -pg+\frac{(p-1)(pn)}{2}$; there is no homology in Alexander grading $-pg+\frac{(p-1)(pn)}{2}$ of the appropriate Maslov grading (provided $K$ is a non-trivial knot).  If $\tau(K)=-g$, then $y^-(n)= -pg+\frac{(p-1)(pn)}{2}$, and hence $\ctauprime{n}=y^+(n)$ for all $n>N$.   Using Theorem $3.8$ of \cite{Cabling} in place of Theorem \ref{thm:largen} above, the same argument shows that if $\tau(K)=g(K)$ then  $\ctauprime{n}=y^-(n)$ for all $n<-N$.  This completes the proof.
\end{proof}

\subsection{Proof of Theorem \ref{thm:lspace}}

 The strategy here is to show that if $K$ is an L-space knot, then surgery on $K_{p,q}$ will be an L-space for $q$ large enough.     We will achieve this through a standard topological argument, together with formulas for the Floer homology of manifolds obtained by Dehn surgery on knots and connected sums, respectively.  More precisely, Theorem \ref{thm:lspace} is an immediate consequence of the following facts:
 \begin{enumerate}
\item $pq$ surgery on $K_{p,q}$ is the manifold $S^3_{q/p}(K)\# L(p,q)$, where $S^3_{q/p}(K)$ is the manifold obtained by $q/p$ Dehn surgery on $K$.
\item The lens space, $L(p,q)$, is an L-space
\item If any positive surgery on $K$ yields an L-space, then $q/p$ surgery on $K$ is an L-space for any $q/p\ge 2g(K)-1$, where $g(K)$ is the genus of $K$.
\item If $Y_1$ and $Y_2$ are L-spaces, then $Y_1\#Y_2$ is an L-space.

\end{enumerate}

The first fact is well-known to those working with Dehn surgery.  For completeness, we include a proof below. The Floer homology of lens spaces can easily be computed from their genus one Heegaard diagram, verifying $(2$).   The third fact follows from a general formula which computes the Floer homology of manifolds obtained by surgery on $K$ in terms of the knot Floer homology invariants \cite{RationalSurgeries,Lens}. Specifically, we have
\begin{lemma} 
\label{lemma:CalcRanks}
Let $K\subset S^3$ be an L-space knot, and fix a pair of relatively prime integers $p$ and $q$.
Then
$$
\Rk \HFa(S^3_{q/p}(K)) =
|q| + 2\max(0,(2g(K)-1)|p|-|q|).
$$
\end{lemma}
\begin{proof} The lemma is a particular case of Proposition $9.5$ of \cite{RationalSurgeries}.   Specifically, if $K$ is an L-space knot, results of \cite{Lens} show that $\tau(K)=g(K)$.  This implies that the term $\nu(K)$ appearing in Proposition $9.5$ of \cite{RationalSurgeries} is equal to the genus, since $\nu(K)$ is equal to $\tau(K)$ or $\tau(K)+1$ by definition, and $\nu(K)\le g(K)$ by the adjunction inequality (Theorem $5.1$ of \cite{Knots}).  The term in Proposition $9.5$ involving rk $H_*(\widehat{A}_s)$ vanishes under the assumption that $K$ is an L-space knot, since in this case rk $H_*(\widehat{A}_s)=1$.
\end{proof}

With the lemma in hand, $(3)$ follows immediately: if  $q/p\ge 2g(K)-1$, the second term in the proposition vanishes and $$\Rk\  \HFa(S^3_{q/p}(K)) =
|q| = |H_1(S^3_{q/p}(K);\Z)|.$$

The last fact follows from a  K{\"u}nneth type formula for the Floer homology of manifolds obtained by connected sum, Theorem $1.5$ of \cite{HolDiskTwo}.  This theorem says that the Floer homology of the connected sum, $Y_1\#Y_2$, can be computed from a chain complex quasi-isomorphic to the tensor product of Floer chain complexes for $Y_1$ and $Y_2$.  In particular, it implies that $$\mathrm{rk}\ \HFa(Y_1\#Y_2)=\mathrm{rk}\ \HFa(Y_1)\cm \mathrm{rk} \ \HFa(Y_2).$$ Thus $(4)$ follows from the definition of an L-space and the elementary observation that $$|H_1(Y_1\#Y_2;\Z)|=|H_1(Y_1;\Z)|\cm |H_1(Y_2;\Z)|.$$

We conclude by showing that $pq$ surgery on $K_{p,q}$ is $S^3_{q/p}(K)\# L(p,q)$.  To see this, decompose $S^3$ as $$S^3=E(K) \underset{T_K} \amalg N(K),$$ where $N(K)$ is a tubular neighborhood of $K$, $E(K)=S^3-N(K)$, and $T_K=\partial N(K)$ (see Figure \ref{fig:cablesurgery}).  $K_{p,q}$ is embedded in $T_K$ as a curve of slope $p/q$. Here, the meridian of $K$ has slope $0/1$, while the longitude has slope $1/0$.  
meridional
Consider next the tubular neighborhood of the cable.  Denote this by $N(K_{p,q})$.   The intersection $A=N(K_{p,q})\cap T_K$ is an annular neighborhood of $K_{p,q}$  in $T_K$.   The boundary of this annulus consists of two parallel copies of $K_{p,q}$, which we denote by $\lambda$ and $\lambda'$, each of which have linking number  $pq$ with $K_{p,q}$.  
    Let us examine the result of surgery on $K_{p,q}$ with framing given by $\lambda$ (or equivalently, $\lambda'$).  The fact that lk$(K_{p,q},\lambda)=pq$ is equivalent to the slope of the surgery being $pq/1$.  Now the exterior of $K_{p,q}$ can be decomposed as $$ E(K_{p,q})= E(K) \underset{T_K-A} \amalg N(K).$$ 
\noindent  (Figure \ref{fig:cablesurgery2}) Since $K_{p,q}$ is an essential curve on $T_K$, we see that $T_K-A$ is an annulus.  Now the surgery is performed by gluing a solid torus $D^2\times S^1$ to $E(K_{p,q})$ in such a way that the boundary of  each meridional disk is  identified with a curve on $\partial N(K_{p,q})$ isotopic to $\lambda$.  This gluing can, equivalently, be thought of as attaching a pair of two-handles (Figure \ref{fig:cablesurgery3}) $H_1=D^2\times [0,\pi]$, $H_2= D^2\times [\pi, 2\pi]$ to $E(K_{p,q})$, so that 
$$ S^3_{pq}(K_{p,q})= [ E(K)\underset{\partial D^2\times [0,\pi] }\sqcup H_1 ] \underset {D^2 \times \{0\} \sqcup (T_K-A) \sqcup D^2\times \{\pi\} } \amalg [H_2 \underset{\partial D^2 \times [\pi,2\pi] }\sqcup N(K)].$$
\noindent Since $K_{p,q}$ is a  $p/q$ curve on $T_K=\partial N(K)$, the handle $H_2$ is attached to the solid torus $N(K)$ along a curve of slope $p/q$.   From the perspective of $E(K)$, however, $K_{p,q}$ is a curve of slope $q/p$.  It follows that term on the left is $S^3_{q/p}(K)-D^3$, while the term on the right is $L(p,q)-D^3$.  These two three-manifolds are  joined along their common ($2$-sphere) boundary, $D^2 \times \{0\} \sqcup (T_K-A) \sqcup D^2\times \{\pi\}$, completing the proof.  \ \ \ \ \ \ \ \ \ \ \ \ \ \ \ \ \ \ \ \ \ \ \ \ \ \ \  \ \ \ \ \ \ \ \ \ \ \ \ \ \ \ \ \ \ \ \ \ \  \ \ \ \ \ \ \ \ \ \ \ \ \ \ \ \ \ \  \ \ \ \ \ \ \ \ \ \ \ \ \ \ \ \ \ \ \ $\square$

\begin{figure}

\psfrag{N(K)}{\tiny $N(K)$}
\psfrag{N(K')}{\small $N(K)$}
\psfrag{K}{\tiny $K$}
\psfrag{S^3}{\small $S^3$}
\psfrag{T}{\tiny $T_K$}
\psfrag{A}{\small $A$}
\psfrag{K'}{\small$K_{p,q}$}
\psfrag{E(K)}{\tiny $E(K)$}

\includegraphics{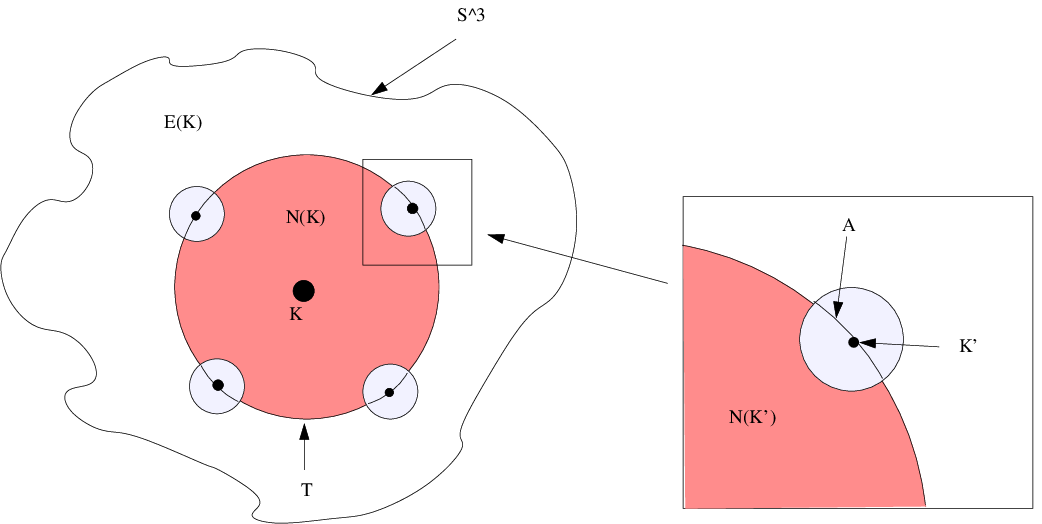}

\caption{\label{fig:cablesurgery}
 }

\end{figure}

\bigskip 
\begin{figure}
\psfrag{E(K)}{\tiny $E(K)$}

\psfrag{N(K)}{\tiny $N(K)$}
\psfrag{N(K')}{\tiny $N(K_{p,q})$}
\psfrag{K}{\tiny $K$}
\psfrag{S^3}{\small $E(K_{p,q})$}
\psfrag{T}{\tiny $T_K\sd A$}
\psfrag{T'}{\small $T_K\sd A$}

\psfrag{K'}{\small$\lambda$}
\psfrag{K''}{\small$\lambda'$}

\includegraphics{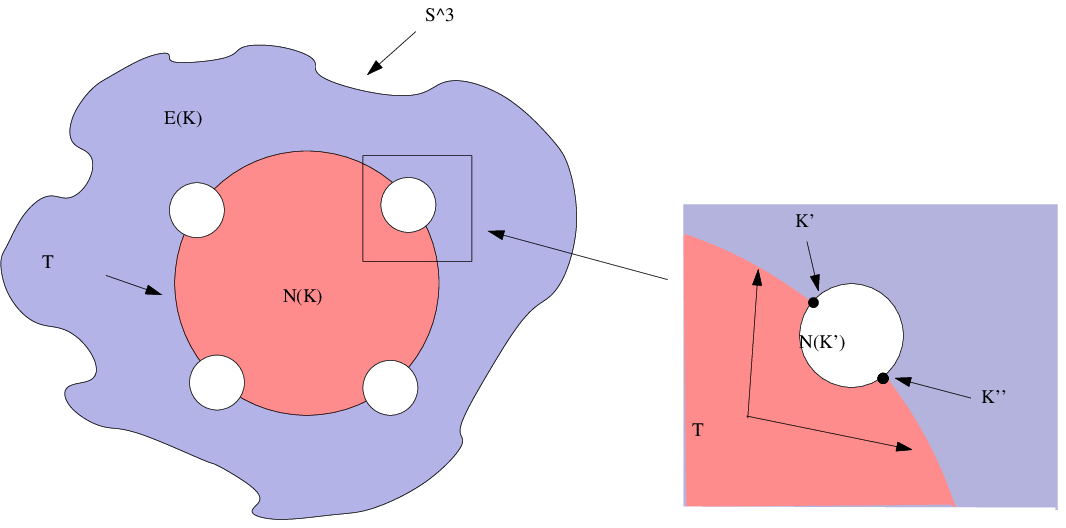}

\caption{\label{fig:cablesurgery2}
 }

\end{figure}

\begin{figure}

\psfrag{D1}{\small $D^2\times\{0\}$}
\psfrag{D2}{\small $D^2\times\{\pi\}$}
\psfrag{H1}{\small $H_1\cong D^2\times[0,\pi]$}
\psfrag{H2}{\small $H_2\cong D^2\times[\pi,2\pi]$}

\includegraphics{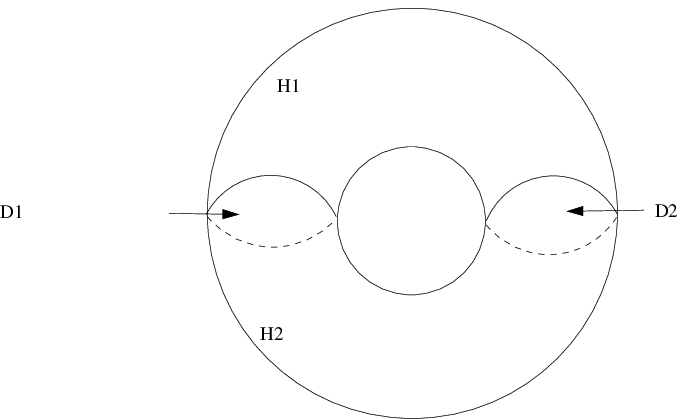}

\caption{\label{fig:cablesurgery3}
 }

\end{figure}

\section{Proof of Corollaries}
\label{sec:cor}
In this section we prove the corollaries stated in the introduction.  The heart of the corollaries is that $\tau(K)$ carries a great deal of geometric information, and thus can be used in conjunction with Theorem \ref{thm:tau} to obstruct cables of $K$ from having certain geometric or braid theoretic properties.

For instance, we can derive several consequences from the second half of Theorem \ref{thm:tau}.  To make this precise, define
 $$\Pos := \{ K\subset S^3 \ | \ \tau(K)=g(K)\}.$$   The following is an immediate corollary of Theorem \ref{thm:tau}:
\begin{cor} \label{cor:pos} \ \ 

 \begin{itemize}
 \item If $K\in \Pos$, then $K_{p,pn+1}\in \Pos$ if and only if $n\ge 0$.
 \item If $K\notin \Pos$, then $K_{p,pn+1}\notin \Pos$ for any $n$.
 \end{itemize}
\end{cor}

This corollary derives its power from the fact that there are several  classes of knots which we know to be contained in $\Pos$.    For instance, the following classes of knots are contained in $\Pos$:
 \begin{enumerate}
 \item Knots which bound a complex curve, $V_f\subset B^4$, satisfying $g(V_f)=g(K)$ \cite{SQPfiber}.
 \item  Positive knots i.e. those knots which admit a diagram containing only positive crossings \cite{Livingston2004}.
 \item L-space knots i.e. knots for which positive slope Dehn surgery on $K$ yields an L-space (in particular, lens space knots)  \cite{Lens}.
 \item Any non-negatively twisted, positive-clasped Whitehead double of a knot satisfying $\tau(K)>0$ \cite{Doubling}.
 \item Strongly quasipositive knots i.e. those knots bounding a Seifert surface obtained  from parallel disks by attaching bands with a positive half twist \cite{Livingston2004}  (see also \cite{SQPfiber}).
 \item Fibered knots whose associated contact structure is tight \cite{SQPfiber}.
 \end{enumerate}
These classes overlap highly.  For instance, $(2)\subset (5)\subset (1)$, and $(3)\subset (6) \subset (5)$  (see \cite{SQPfiber} for a discussion of these inclusions.)     
Combining Corollary \ref{cor:pos} with $(1)$ yields Corollary \ref{cor:cknot} of the introduction.  Combining with $(3)$ yields Corollary \ref{cor:cor2}, and $(6)$ is instrumental in the results of \cite{ComplexCable}.  Item $(2$) yields an obstruction for cabling to produce positive knots:

\begin{cor} Suppose $K\notin \Pos$.  Then $K_{p,pn+1}$ is not a positive knot for any $n\in \Z$ \end{cor}

Item $(5)$ produces the most precise information to date on the smooth four-genera of knots obtained by iterated doubling and cabling:

\begin{cor} Suppose $\tau(K)>0$.  Then any knot, $S$, obtained by an arbitrary sequence of non-negative cabling and Whitehead double operations is in $\Pos$.  In particular, $g_4(S)=g(S)$.
\end{cor}

We have seen that understanding when  $K_{p,pn+1}\in \Pos$ has geometric consequences for cabling.  Likewise, understanding  when $\tau(K_{p,pn+1})\ge 0$ is also tied to geometry.  In this case, Theorem \ref{thm:tau} yields:

\begin{cor}
Suppose $\tau(K_{p,pn+1})\ge 0$.  Then $n\ge -2\left(\frac{\tau(K)}{p-1}+ \frac{1}{p}\right)$.
\end{cor}

This is relevant in light of the connection between $\tau(K)$ and complex curves.   Suppose that $K$ is a $\C$-knot.  Then results of  \cite{Olga2004,SQPfiber,Rudolph1983}  show that $\tau(K)=g_4(K)$.   In particular, $\tau(K)\ge 0$.  Combining this fact with the corollary yields Corollary \ref{cor:cor1} of the introduction.  

Corollary \ref{cor:cor1} could alternatively be stated in terms of the braid group.  Let $B_n$ denoted the braid group on $n$ strands, with generators $\sigma_1,\ldots,\sigma_{n-1}$.  A {\em quasipositive}  knot is any knot which can be realized as the closure of a braid of the form:
$$\beta = \Pi_{k=1}^m w_k \sigma_{i_k} w_k^{-1}.$$

 Results of \cite{Rudolph1983} and \cite{BO2001} indicate that $\C$-knots are exactly the class of knots obtained as the closures of quasipositive braids.  Thus, we obtain
 
 \begin{cor}
 Suppose $K_{p,pn+1}$ is the closure of a quasipositive braid.  Then $$n\ge -2(\frac{\tau(K)}{p-1}+ \frac{1}{p}).$$
 \end{cor}

\bigskip

\noindent We conclude with a proof of Corollary \ref{cor:concordance}. 
 
\bigskip 
\noindent{\bf Proof of Corollary \ref{cor:concordance}.} We wish to show that $\phi_{p,1}$ is not a homomorphism, so we must find knots $K_1$ and $K_2$ for which $$\phi_{p,1}([K_1\#K_2])\ne [\phi_{p,1}(K_1)\#\phi_{p,1}(K_2)].$$  To do this, it suffices to show that $$\tau^p(K_1\#K_2,1)\ne \tau^p(K_1,1) + \tau^p(K_2,1),$$   This is accomplished with $$K_1=\mathrm{right\sd handed \ trefoil}$$ $$K_2=\mathrm{left\sd handed \ trefoil}$$
   $K_1\#K_2$ is slice, and so $\tau^p(K_1\#K_2,1)=\tau^p(U,1)=0,$ where $U$ is the unknot.  Now $\tau(K_1)=g(K_1)=1$, so Theorem \ref{thm:tau} implies $\tau^p(K_1,1)=p\tau(K_1)=p$.   For $K_2$, however, we have $\tau(K_2)=-g(K_2)=-1$.  Thus $\tau^p(K_2,1)=p\tau(K_2)+p-1=-1$ (again by Theorem \ref{thm:tau}).  This completes the proof.   \ \ \ \ \ \ \ \ \ \ \ \ \ \ \ \ \ \ \ \ \ \ \ \ \ \ \  \ \ \ \ \ \ \ \ \ \ \ \ \ \ \ \ \ \ \ \ \ \  \ \ \ \ \ \ \ \ \ \ \ \ \ \ \ \ \ \ \ \ \ \ \ \ \ \ \ \ \ \ \ \ \ \ \ \ \ \ \ \ \  \ \ \ \ \ \ \ \ \ \ \ \ \ \ $\square$

\bibliographystyle{plain}
\bibliography{mybib}

\end{document}